\documentclass[draft]{amsart}
\usepackage{amssymb,graphicx}
\usepackage{multirow,comment,caption}

\newcommand{\df}{\dfrac}
\newcommand{\tf}{\tfrac}

\newcommand{\beqs}{\begin{equation*}}
\newcommand{\eeqs}{\end{equation*}}
\numberwithin{equation}{section}
 \theoremstyle{plain}
\newtheorem{theorem}{Theorem}[section]

\newtheorem{lemma}[theorem]{Lemma}
\newtheorem{ident}[theorem]{Identity}

\theoremstyle{remark}

\begin{document}

\makeatletter
\def\imod#1{\allowbreak\mkern10mu({\operator@font mod}\,\,#1)}
\makeatother

\author{Frank Patane}
   \address{Department of Mathematics, Samford University, 800 Lakeshore Drive, Birmingham AL 35209, USA}
   \email{fpatane@samford.edu}

\title[\scalebox{.82}{On a generalized identity connecting theta series associated with discriminants $\Delta$ and $\Delta p^2$}]{On a generalized identity connecting theta series associated with discriminants $\Delta$ and $\Delta p^2$}
     
\begin{abstract} 
When the discriminants $\Delta$ and $\Delta p^2$ have one form per genus, \cite{patane} proves a theorem which connects the theta series associated to binary quadratic forms of each discriminant. This paper generalizes the main theorem of \cite{patane} by allowing $\Delta$ and $\Delta p^2$ to have multiple forms per genus. In particular, we state and prove an identity which connects the theta series associated to a single binary quadratic form of discriminant $\Delta$ to a theta series associated to a subset of binary quadratic forms of discriminant $\Delta p^2$. Here and everywhere $p$ is a prime.
\end{abstract}



\date{\today}
   
\maketitle

   \section{Introduction}
\label{intro}

Let $\Delta$ be a discriminant of a positive definite binary quadratic form. When the discriminants $\Delta$ and $\Delta p^2$ have one form per genus, \cite{patane} gives an identity that connects the theta series associated to binary quadratic forms for each discriminant. This paper is mainly concerned with generalizing the central identity of \cite{patane} to discriminants which have multiple forms per genus. This generalized identity is stated in Theorem \ref{newthm} where the discriminants $\Delta$ and $\Delta p^2$ are not required to have one form per genus. Theorem \ref{newthm} gives an identity which connects a theta series associated to a binary quadratic form of discriminant $\Delta$ to a theta series associated to a subset of binary quadratic forms of discriminant $\Delta p^2$.

 Section \ref{bmap} sets the notation and discusses some preliminary results. Section \ref{bue} considers a map of Buell which connects the class groups CL$(\Delta)$ and CL$(\Delta p^2)$. Section \ref{lemi} contains the lemmas and identities which are necessary for the proof of Theorem \ref{newthm}. Section \ref{main} combines the results of the previous sections to prove Theorem \ref{newthm}. Section \ref{conc} employs Theorem \ref{newthm} to prove a general theorem given by \cite[Theorem 5.1]{patane}. Lastly, Section \ref{examples} gives an explicit example which employs Theorem \ref{newthm} to derive a Lambert series decomposition and the corresponding product representation formula.

\section{Preliminaries and Notation}

\label{bmap}

We use $(a,b,c)$ to represent the class of binary quadratic forms which are equivalent to the binary quadratic form $ax^2 +bxy +cy^2$. Equivalence of two binary quadratic forms means the transformation matrix which connects them is in $SL(2,\mathbb{Z})$. The discriminant of $(a,b,c)$ is defined as $\Delta:= b^2-4ac$, and we only consider the case $\Delta <0$ and $a>0$. We say $(a,b,c)$ is primitive when GCD($a,b,c)=1$. The set of all classes of primitive forms of discriminant $\Delta$ comprise what is known as the class group of discriminant $\Delta$, denoted CL$(\Delta)$. We will often use the term ``form'' to mean a class of binary quadratic forms.

We use $h(\Delta):=|$CL$(\Delta)|$ to denote the class number of $\Delta$. In the 1801 work, Disquisitiones Arithmeticae, Gauss develops much of the theory of binary quadratic forms, including the below relation between $\Delta$ and $\Delta p^2$ \cite{gauss}:
\begin{equation}
\label{hrel}
h(\Delta p^2)=\df{h(\Delta)\left(p-\left(\tf{\Delta}{p}\right)\right)}{w},
\end{equation}
where
\begin{equation}
\label{www}
	w :=\left\{ \begin{array}{ll}
        3&  \Delta  =-3,\\
				   2&  \Delta =-4,\\
					   1&  \Delta  <-4.\\
     \end{array}
     \right.
	\end{equation}
The relation \eqref{hrel} as well as the number $w$ in \eqref{www} appear in Section \ref{bmap}. Two binary quadratic forms of discriminant $\Delta$ are said to be in the same genus if they are equivalent over $\mathbb{Q}$ via a transformation matrix in $SL(2,\mathbb{Q})$ whose entries have denominators coprime to $2\Delta$. An equivalent definition for the genera of binary quadratic forms is given by introducing the concept of assigned characters. The assigned characters of a discriminant $\Delta$ are the functions $\left(\tf{r}{p}\right)$ for all odd primes $p\mid \Delta$, as well as possibly the functions $\left(\tf{-1}{r}\right)$, $\left(\tf{2}{r}\right)$, and $\left(\tf{-2}{r}\right)$. The details are discussed in Buell \cite{buell} and in Cox \cite{cox}.\\

The genera are of equal size and partition the class group. We say a discriminant is idoneal when each genus contains only one form. The number of genera of discriminant $\Delta p^2$ is either equal to the number of genera of discriminant $\Delta$ or double the number of genera of discriminant $\Delta$. Letting $v(\Delta)$ be the number of genera of discriminant $\Delta$ we have

	\begin{equation}
\label{numgenform}
\frac{v(\Delta p^2)}{v(\Delta)} = \left\{ \begin{array}{ll}
        1&  2<p, p\mid\Delta,\\
				2&  2<p, p\nmid\Delta,\\
				1&  p=2, p\nmid\Delta,\\
				1&  p=2, \Delta \equiv 0,12,28 \imod{32},\\
				2&  p=2, \Delta \equiv 4,8,16,20,24 \imod{32}.\\
     \end{array}
     \right.
\end{equation}
The theta series associated to $(a,b,c)$ is
\[
(a,b,c,q):=\sum_{x,y}q^{ax^2+bxy+cy^2}=\sum_{n\geq0}(a,b,c;n)q^n,
\]
where we use $(a,b,c;n)$ to denote the total number of representations of $n$ by $(a,b,c)$. We define the projection operator $P_{m,r}$ to be 
\[
P_{m,r}\sum_{n \geq 0}a(n)q^n = \sum_{n\geq 0}a(mn+r)q^{mn+r},
\]
where we take $0\leq r<m$. Informally, the operator $P_{p,0}$ applied to $(a,b,c,q)$ collects the terms of $(a,b,c,q)$ which have the exponent of $q$ congruent to $0\imod{p}$

\section{Connecting $\Delta$ to $\Delta p^2$}

\label{bue}

Let $(a,b,c)$ be a primitive form of discriminant $\Delta$. In Chapter 7 of \cite{buell}, Buell defines a map which sends $(a,b,c)$ to a set of $p+1$ not necessarily distinct and not necessarily primitive forms of discriminant $\Delta p^2$. The image of $(a,b,c)$ under this map is given by
\begin{equation}
\label{buellh}
\{(a,bp,cp^2)\} \cup \{(ap^2, pb+2ahp, ah^2 +bh+c): 0\leq h<p \}.
\end{equation}
Buell devotes Section 7.1 of his book to determine the important properties of this map \cite[Section 7.1]{buell}. Buell shows that if we cast out the nonprimitive forms of \eqref{buellh}, then the remaining forms (all primitive) are repeated $w$ times, where $w$ is half the number of automorphs of $\Delta$ and is given in \eqref{www}. We can map $(a,b,c)$ to the set of distinct primitive forms of \eqref{buellh}, and we call this set $\Psi_p(a,b,c)$.

 Buell shows the images of $\Psi_p$ are distinct, of the same size, and partition CL$(\Delta p^2)$ \cite[Section 7.1]{buell}. Moreover, there are exactly $1+\left(\tf{\Delta}{p}\right)$ nonprimitive forms in \eqref{buellh}. Thus there are $p+1-\left(1+\left(\tf{\Delta}{p}\right)\right) =p-\left(\tf{\Delta}{p}\right) $ primitive forms in \eqref{buellh}. In other words, $|\Psi_p(a,b,c)|=\tf{p-\left(\tf{\Delta}{p}\right)}{w}$. Combining these results, Buell derives the class number of $\Delta p^2$ to be
\begin{equation}
\label{hrel2}
h(\Delta p^2)=\df{h(\Delta)\left(p-\left(\tf{\Delta}{p}\right)\right)}{w}.
\end{equation}
We emphasize that the only time there are repeated primitive forms in \eqref{buellh} is when $\Delta=-3,-4$. As an example we take $\Delta=-3$ and $p=7$. The class group for $\Delta=-3$ consists of the single reduced form $(1,1,1)$. The class group for discriminant $\Delta p^2=-3\cdot 7^2=147$ consists of the two reduced forms $(1,1,37)$ and $(3,3,13)$. The forms in \eqref{buellh} (counting repetition) consist of $(1,1,37)$ union the forms listed in Table \ref{tt}.

\begin{table}
\caption{$\Delta =-3$, $p=7$}
\label{tt}
\begin{center}
\begin{tabular}{ |l |c| }
  \hline
	$h$ & Corresponding form of \eqref{buellh}  \\ \hline
  0 & $(1,1,37)$ \\ \hline
	1 & $(3,3,13)$ \\ \hline
	2 & $(7,7,7)$ \\ \hline
	3 & $(3,3,13)$ \\ \hline
	4 & $(7,7,7)$ \\ \hline
	5 & $(3,3,13)$ \\ \hline
	6 & $(1,1,37)$ \\ \hline
\end{tabular}
\end{center}
\end{table}

As expected, the primitive forms are repeated $w=3$ times and there are $1+\left(\tf{-3}{7}\right)=2$ nonprimitive forms. We have $\Psi_7(1,1,1) = \{(1,1,37), (3,3,13)\}$. The preceding example illustrates the map $\Psi_p$ when $w>1$ and when \eqref{buellh} contains nonprimitive forms. If we apply Theorem \ref{newthm} to this example we obtain identities which are discussed in \cite{patane}.

As another example we let $\Delta=-55$ and $p=3$. The genus structure along with the assigned characters for the genera for the discriminants $\Delta=-55$ and $\Delta p^2=-495$ are given below.

\begin{center}
\begin{tabular}{ | l | l | l | l | }
  \hline     
  \multicolumn{2}{|c|}{CL$(-55) \cong \mathbb{Z}_4$}& $\left(\tf{r}{5}\right)$ & $\left(\tf{r}{11}\right)$\\
  \hline                   
   $g_1$ & $(1,1,14)$, $(4,3,4)$ &$+1$ & $+1$ \\ \hline 
	$g_2$ & $(2,1,7)$, $(2,-1,7)$ &$-1$ & $-1$ \\ \hline 
\end{tabular}
\begin{flushright}
			 .
			 \end{flushright}
   \end{center}

\begin{center}
\begin{tabular}{ | l | l | l | l | l| }
  \hline     
  \multicolumn{2}{|c|}{CL$(-495) \cong \mathbb{Z}_8 \times \mathbb{Z}_2 $}& $\left(\tf{r}{3}\right)$ & $\left(\tf{r}{5}\right)$ & $\left(\tf{r}{11}\right)$\\
  \hline                   
   $G_1$ & $(1,1,124)$, $(9,9,16)$, $(4,1,31)$, $(4,-1,31)$, &$+1$ & $+1$ & $+1$\\ \hline 
	$G_2$ & $(5,5,26)$, $(11,11,14)$, $(9,3,14)$, $(9,-3,14)$, &$-1$ & $+1$ & $+1$\\ \hline 
	$G_3$ & $(2,1,62)$, $(2,-1,62)$, $(8,7,17)$, $(8,-7,17)$, &$-1$ & $-1$ & $-1$\\ \hline 
	$G_4$ & $(7,3,18)$, $(7,-3,18)$, $(10,5,13)$, $(10,-5,13)$, &$+1$ & $-1$ & $-1$\\ \hline 
\end{tabular}
\begin{flushright}
			 .
			 \end{flushright}
   \end{center}

We compute
	\begin{equation}
	\label{psihelp3}
	\begin{aligned}
  \Psi_{3}(1,1,14)&=\{(1,1,124), (9,9,16), (9,3,14), (9,-3,14)\},\\
	\Psi_{3}(4,3,4)&=\{(5,5,26), (11,11,14), (4,1,31), (4,-1,31)\},\\
	\Psi_{3}(2,1,7)&=\{(2,-1,62), (7,-3,18), (8,-7,17), (10,5,13)\},\\
	\Psi_{3}(2,-1,7)&=\{(2,1,62), (7,3,18), (8,7,17), (10,-5,13)\}.\\
	\end{aligned}
	\end{equation}
Also we see that
\begin{equation}
	\label{psihelp4}
	\begin{aligned}
  \Psi_{3}(g_1)&=G_1 \cup G_2,\\
	\Psi_{3}(g_2)&=G_3 \cup G_4.\\
	\end{aligned}
	\end{equation}

As expected, the images are distinct, of equal size, and partition CL$(\Delta p^2)$. Also we see in this example that $\Psi_p(f)$ is split evenly over two genera, and doesn't necessarily contain a form and its inverse. In general, the set $\Psi_p(f)$ will either be fully contained in one genus or be split equally between two genera. This behavior corresponds to whether $\tf{v(\Delta p^2)}{v(\Delta)}=1,2$, respectively. We refer the reader to \eqref{numgenform} for the cases.

\section{Lemmas and Identities}
\label{lemi}

This section contains several lemmas and identities which we use to prove Theorem \ref{newthm}. Lemma \ref{nplem} shows exactly which forms in \eqref{buellh} are nonprimitive.

\begin{lemma}
\label{nplem}
Let $(a,b,c)$ be a primitive form of discriminant $\Delta$. There are exactly $1+\left(\tf{\Delta}{p}\right)$ nonprimitive forms in the list 
\begin{equation}
\label{buellh2}
\{(a,bp,cp^2)\} \cup \{(ap^2, pb+2ahp, ah^2 +bh+c): 0\leq h<p \},
\end{equation}
and the nonprimitive forms are given by
\begin{equation}
\label{npp0}
\mbox{non-primitive forms of \eqref{buellh2}}= \left\{ \begin{array}{ll}
        (a,bp,cp^2),&  p\mid a,~\left(\tf{\Delta}{p}\right)=0,\\
				f_1,&  p\nmid a,~\left(\tf{\Delta}{p}\right)=0,\\
				\emptyset,&  \left(\tf{\Delta}{p}\right)=-1,\\
				(a,pb,cp^2), f_2,&  p\mid a,~\left(\tf{\Delta}{p}\right)=1,\\
				f_3,f_4,&  p\nmid a,~\left(\tf{\Delta}{p}\right)=1,\\
     \end{array}
     \right.
\end{equation}
where $f_i:=(ap^2, p(b+2ah_i), ah_{i}^2 + bh_i +c)$. For $p$ odd we take
\begin{align}
h_1 \equiv \tf{-b}{2a} \imod{p},\\
h_2\equiv \tf{-c}{b} \imod{p},\\
h_3\equiv \tf{-b+\sqrt{\Delta}}{2a} \imod{p},\\
h_4\equiv \tf{-b-\sqrt{\Delta}}{2a} \imod{p},
\end{align}
and for $p=2$ we take
$h_4\not\equiv h_3 \equiv h_2 \equiv h_1 \equiv c \imod{2}$. We always take $0\leq h_i <p$.\\
\end{lemma}

\begin{proof}
Let $(a,b,c)$ be a primitive form of discriminant $\Delta <0$. Since $(a,b,c)$ is assumed primitive, the only way for a form in \eqref{buellh2} to be nonprimitive is if $p$ divides each of its entries. Hence $(a,bp,cp^2)$ is nonprimitive if and only if $p\mid a$. This takes care of the form $(a,bp,cp^2)$, and we are left with considering the forms $(ap^2, p(b+2ah), ah^2 + bh +c)$, with $0\leq h <p$.\\

The form $(ap^2, p(b+2ah), ah^2 + bh +c)$ is nonprimitive if and only if $p\mid (ah^2 + bh +c)$, and so the remainder of the proof is devoted to determining exactly when $ah^2 + bh +c \equiv 0 \imod{p}$.\\

\noindent
\large\textbf{Case $p\mid a$:}

If $p \mid a$ and $p\mid b$, then $ah^2 + bh +c \equiv 0 \imod{p}$ has no solutions since $(a,b,c)$ is primitive. If $p \mid a$ and $p\nmid b$, then $ah^2 + bh +c \equiv 0 \imod{p}$ has the unique solution $h\equiv \tf{-c}{b} \imod{p}$, $0\leq h <p$. We note that when $p\mid a$ and $p \nmid b$, we have $\Delta \equiv b^2 \imod{p}$ and so $\left(\tf{\Delta}{p}\right)=1$. We have found the form $f_2$ in \eqref{npp0}.\\

\noindent
\large\textbf{Case $p\nmid a$, $p\neq 2$:}

Since $p\nmid a$, we see
\[
ah^2 + bh +c \equiv 0 \imod{p}
\]
is equivalent to 
\[
(2ah+b)^2 \equiv \Delta \imod{p}.
\]
We find the following forms are nonprimitive: 
\begin{equation}
 \left\{ \begin{array}{ll}
				f_1,&  p\nmid a,~\left(\tf{\Delta}{p}\right)=0,\\
				\emptyset,&  p\nmid a~\left(\tf{\Delta}{p}\right)=-1,\\
				f_3,f_4,&  p\nmid a,~\left(\tf{\Delta}{p}\right)=1,\\
     \end{array}
     \right.
\end{equation}
where $f_i:=(ap^2, p(b+2ah_i), ah_{i}^2 + bh_i +c)$ with
\begin{align}
h_1 \equiv \tf{-b}{2a} \imod{p},\\
h_3\equiv \tf{-b+\sqrt{\Delta}}{2a} \imod{p},\\
h_4\equiv \tf{-b-\sqrt{\Delta}}{2a} \imod{p},
\end{align}
 and $0\leq h_1,h_2,h_3 <p$.\\

\noindent
\large\textbf{Case $p\nmid a$, $p=2$:}

Since $2\nmid a$, we have
\[
ah^2 + bh +c \equiv h + bh +c \equiv (b+1)h+c \imod{2}.
\]

If $\left(\tf{\Delta}{2}\right)=0$, then $2 \mid b$ and we have
\[
(b+1)h+c \equiv 0 \imod{2}
\]
implies $h\equiv c \imod{2}$. We have arrived at the nonprimitive form $f_1$  with $h_1 \equiv c \imod{2}$.\\

If $\left(\tf{\Delta}{2}\right)=-1$ then $2 \nmid b$ and we have
\[
\Delta \equiv  1-4ac \equiv 5 \imod{8},
\]
and so $c$ is odd in this sub-case. Thus $a,b,c$ are all odd and we see $(4a,2b,c)$ and $(4a,6b,a+b+c)$ are primitive. In other words, $(ap^2, pb+2ahp, ah^2 +bh+c)$ with $h=0,1$ are both primitive forms. Hence we have only primitive forms in this sub-case. \\

If $\left(\tf{\Delta}{p}\right)=1$ then $2 \nmid b$ and we have
\[
\Delta \equiv 1-4ac \equiv 1 \imod{8}
\]
so that $c$ is even. Thus $a,b$ are odd and $c$ is even implies both $(4a,2b,c)$ and $(4a,6b,a+b+c)$ are nonprimitive. Hence $(ap^2, pb+2ahp, ah^2 +bh+c)$ with $h=0,1$ are both nonprimitive forms.\\

We now list the nonprimitve forms found in this case: 
\begin{equation}
 \left\{ \begin{array}{ll}
				f_1,&  2\nmid a,~\left(\tf{\Delta}{2}\right)=0,\\
				\emptyset,&  2\nmid a~\left(\tf{\Delta}{2}\right)=-1,\\
				f_3,f_4,&  2\nmid a,~\left(\tf{\Delta}{2}\right)=1,\\
     \end{array}
     \right.
\end{equation}
where $f_i:=(ap^2, p(b+2ah_i), ah_{i}^2 + bh_i +c)$ with
\[
h_4\not\equiv h_3 \equiv h_1 \equiv c \imod{2},
\]
and $0\leq h_i<2$.\\

We have considered all possible cases and completed the proof of Lemma \ref{nplem}.
\end{proof}

Lemma \ref{nplem} is essential to finding which forms are in $\Psi_p(a,b,c)$, and we are a step closer to proving Theorem \ref{newthm}. Before proving Theorem \ref{newthm} we first consider $P_{p,0}(a,b,c,q)$ for an arbitrary primitive form $(a,b,c)$ and prime $p$.

\begin{lemma}
\label{pp0thm}
Let $(a,b,c)$ be a primitive form of discriminant $\Delta$. We have
\begin{equation}
\label{pp0}
P_{p,0}(a,b,c,q)= \left\{ \begin{array}{ll}
        (a,bp,cp^2,q),&  p\mid a,~\left(\tf{\Delta}{p}\right)=0,\\
				f_1(q),&  p\nmid a,~\left(\tf{\Delta}{p}\right)=0,\\
				(a,b,c,q^{p^2}),&  \left(\tf{\Delta}{p}\right)=-1,\\
				f_2(q)+(a,pb,cp^2,q)-(a,b,c,q^{p^2}),&  p\mid a,~\left(\tf{\Delta}{p}\right)=1,\\
				f_3(q)+f_4(q)-(a,b,c,q^{p^2}),&  p\nmid a,~\left(\tf{\Delta}{p}\right)=1,\\
     \end{array}
     \right.
\end{equation}
where $f_i(q):=(ap^2, p(b+2ah_i), ah_{i}^2 + bh_i +c,q)$. For $p$ odd we take
\begin{align}
h_1 \equiv \tf{-b}{2a} \imod{p},\\
h_2\equiv \tf{-c}{b} \imod{p},\\
h_3\equiv \tf{-b+\sqrt{\Delta}}{2a} \imod{p},\\
h_4\equiv \tf{-b-\sqrt{\Delta}}{2a} \imod{p},
\end{align}
and for $p=2$ we take
$h_4\not\equiv h_3 \equiv h_2 \equiv h_1 \equiv c \imod{2}$. We always take $0\leq h_i <p$.
\end{lemma}

\begin{proof} 
The proof is split into cases.

\noindent
\large\textbf{Case $p \mid a$:}\\

If $p\mid a$ and $p\mid \Delta$, then $p\mid b$ and $p \nmid c$ since $(a,b,c)$ is assumed primitive. The congruence 
\[
ax^2+bxy+cy^2\equiv cy^2 \equiv 0 \imod{p},
\]
implies $y \equiv 0 \imod{p}$, and we find $P_{p,0}(a,b,c,q)=(a,bp,cp^2,q)$.\\

If $p\mid a,~p\nmid\Delta$ then $\Delta \equiv b^2 \not\equiv 0 \imod{p}$ and so we must have $\left(\tf{\Delta}{p}\right)=1$. Then
\[
ax^2+bxy+cy^2\equiv y(bx+cy) \equiv 0 \imod{p},
\]
if and only if either $y \equiv 0 \imod{p}$ or $x \equiv \tf{-c}{b}y \imod{p}$. We have
\begin{align*}
P_{p,0}\sum_{x,y}q^{ax^2+bxy+cy^2}&=\sum_{\substack{x,\\y\equiv 0 \imod{p}}}q^{ax^2+bxy+cy^2}+\sum_{\substack{x\equiv \tf{-c}{b}y \imod{p},\\y\not\equiv 0 \imod{p}}}q^{ax^2+bxy+cy^2}\\
&=(a,pb,cp^2,q)+\sum_{x\equiv \tf{-c}{b}y \imod{p}}q^{ax^2+bxy+cy^2}-(a,b,c,q^{p^2})\\
&=(a,pb,cp^2,q)+(ap^2, p(b+2ah_2), ah_{2}^2 + bh_2 +c,q)-(a,b,c,q^{p^2}),
\end{align*}
where $h_2\equiv \tf{-c}{b} \imod{p}$ and Identity \ref{fident3} is employed in the last equality.\\

\noindent
\large\textbf{Case $p\nmid a$, $p\neq 2$:}\\

In this case, the congruence
\[
ax^2+bxy+cy^2\equiv 0 \imod{p},
\]
is equivalent to
\begin{equation}
\label{sss}
(2ax+by)^2 \equiv \Delta y^2 \imod{p}.
\end{equation}

If $\left(\tf{\Delta}{p}\right)=0$ then \eqref{sss} along with Identity \ref{fident3} implies
\[
P_{p,0}(a,b,c,q)=\sum_{x\equiv h_1 y \imod{p}}q^{ax^2+bxy+cy^2}=(ap^2, p(b+2ah_1), ah_{1}^2 + bh_1 +c,q),
\]
where $h_1 \equiv \tf{-b}{2a} \imod{p}$.\\

If $\left(\tf{\Delta}{p}\right)=1$ then \eqref{sss} along with Identity \ref{fident3} yields
\[
P_{p,0}(a,b,c,q)=f_3(q) +f_4(q)-(a,b,c,q^{p^2}),
\]
where 
\begin{align*}
f_3(q)&=(ap^2, p(b+2ah_3), ah_{3}^2 + bh_3 +c,q)\\
f_4(q)&=(ap^2, p(b+2ah_4), ah_{4}^2 + bh_4 +c,q)\\
\end{align*}
and $h_3\equiv \tf{-b+\sqrt{\Delta}}{2a} \imod{p}$, $h_4\equiv \tf{-b-\sqrt{\Delta}}{2a} \imod{p}$.\\

Lastly we note that if $\left(\tf{\Delta}{p}\right)=-1$, then the only solutions to \eqref{sss} is $x\equiv y\equiv 0 \imod{p}$ and hence we have the theorem in this case. We have now finished the proof for $p$ odd.\\

\noindent
\large\textbf{Case $p\nmid a$, $p=2$:}\\

If $\left(\tf{\Delta}{2}\right)=0$ then $2\mid b$ and we have
\[
ax^2 +bxy +cy^2 \equiv x +cy \equiv 0 \imod{2},
\]
implies $x \equiv c y \imod{2}$ is the only solution. Employing Identity \ref{fident3} gives
\[
P_{p,0}(a,b,c,q)=\sum_{x\equiv h_1 y \imod{2}}q^{ax^2+bxy+cy^2}=(ap^2, p(b+2ah_1), ah_{1}^2 + bh_1 +c,q),
\]
where $h_1 \equiv c \imod{2}$.

If $\left(\tf{\Delta}{2}\right)=-1$ then $2\nmid b$ and $\Delta \equiv 1-4ac \equiv 5 \imod{8}$. Thus $c$ is odd and we have
\[
ax^2 +bxy +cy^2 \equiv x +xy+y \equiv 0 \imod{2},
\]
implies $x\equiv y \equiv 0 \imod{2}$ is the only solution, and we have finished this case.\\

If $\left(\tf{\Delta}{2}\right)=1$ then $2\nmid b$ and $\Delta \equiv 1-4ac \equiv 1 \imod{8}$. Thus $c$ is even and we have
\[
ax^2 +bxy +cy^2 \equiv x +xy \equiv 0 \imod{2},
\]
implies $x\equiv 0 \imod{2}$ is a solution or $y\equiv 1 \imod{2}$ is a solution. We find
\begin{align*}
P_{p,0}(a,b,c,q)&=(a,b,c,q^4) + \sum_{\substack{x\equiv 0\imod{2},\\y\not\equiv 0 \imod{2}}}q^{ax^2+bxy+cy^2}+\sum_{\substack{x\not\equiv 0\imod{2},\\y\not\equiv 0 \imod{2}}}q^{ax^2+bxy+cy^2}\\
&=(a,b,c,q^4) +\sum_{\substack{x,\\y\not\equiv 0 \imod{2}}}q^{ax^2+bxy+cy^2}.
\end{align*}
Employing Identity \ref{fident2} finishes the case, and hence the theorem.

\end{proof}

We now state and prove some identities which will be of use in our proof of Theorem \ref{newthm}.

\begin{ident}
\label{fident}
Let $(a,b,c)$ be a primitive form and $0 \leq h <p$. Then
\begin{equation}
\label{funident}
\sum_{\substack{x\equiv 0\imod{p},\\y\equiv j \imod{p}}}q^{ax^2 + (b+2ah)xy + (ah^2 +bh +c)y^2}=\sum_{\substack{x\equiv hj \imod{p},\\y\equiv j \imod{p}}}q^{ax^2 + bxy + cy^2}.
\end{equation}
\end{ident}
\begin{proof}
Use the change of variables $(x,y)\to (x-hy,y)$.
\end{proof}

\begin{ident}
\label{fident2}
Let $(a,b,c)$ be a primitive form. Then
\begin{equation}
\label{funident2}
\sum_{h=0}^{p-1}\sum_{\substack{x\equiv 0\imod{p},\\y\not\equiv 0 \imod{p}}}q^{ax^2 + (b+2ah)xy + (ah^2 +bh +c)y^2}=\sum_{\substack{x,\\y\not\equiv 0 \imod{p}}}q^{ax^2 + bxy + cy^2}.
\end{equation}
\end{ident}

\begin{proof}
Sum \eqref{funident} over $h=0,1, \ldots p-1$ and over $j=1,2,\ldots p-1$. Explicitly one gets
\begin{align}
\sum_{h=0}^{p-1}\sum_{j=1}^{p-1}\sum_{\substack{x\equiv 0\imod{p},\\y\equiv j \imod{p}}}q^{ax^2 + (b+2ah)xy + (ah^2 +bh +c)y^2}\label{a}
&=\sum_{h=0}^{p-1}\sum_{j=1}^{p-1}\sum_{\substack{x\equiv hj\imod{p},\\y\equiv j \imod{p}}}q^{ax^2 + bxy + cy^2}\\
&=\sum_{\substack{x,\\y\not\equiv 0 \imod{p}}}q^{ax^2 + bxy + cy^2}.\label{aa}
\end{align}
\end{proof}

\begin{ident}
\label{fident3}
Let $(a,b,c)$ be a primitive form and $0\leq h <p$. Then
\begin{equation}
\label{funident3}
\sum_{\substack{x\equiv 0\imod{p},\\y}}q^{ax^2 + (b+2ah)xy + (ah^2 +bh +c)y^2} = \sum_{x\equiv hy\imod{p}}q^{ax^2 + bxy + cy^2}.
\end{equation}
\end{ident}
\begin{proof}
Sum \eqref{funident} over $j=0,1,\ldots p-1$. Alternatively one may apply the change of variables $(x,y)\to (x-hy,y)$ directly.
\end{proof}

\begin{lemma}
\label{plemma}
Let $(A,B,C)\in$CL$(\Delta p^2)$. Then
\begin{equation}
\label{cold}
P_{p,0}(A,B,C,q)=(a,b,c,q^{p^{2}}),
\end{equation}
where $(a,b,c)\in$ CL$(\Delta)$ and $(A,B,C) \in \Psi_p(a,b,c)$.
\end{lemma}

\begin{proof}
By Section \ref{bue} we know $(A,B,C)\in$CL$(\Delta p^2)$ implies there exists a unique $(a,b,c)\in$ CL$(\Delta)$ with $(A,B,C) \in \Psi_p(a,b,c)$. In other words, $(A,B,C)$ is equivalent to either $(a,bp,cp^2)$ or to $(ap^2, p(b+2ah), ah^2 +bh+c)$ for some $0\leq h<p$. Applying Lemma \ref{pp0thm} along with Identity \ref{fident3} completes the proof in both cases.
\end{proof}

\section{Statement and Proof of Theorem \ref{newthm}}

\label{main}

We have arrived at the main theorem of our paper.

\begin{theorem}
	\label{newthm}
	Let $(a,b,c)$ be a primitive form of discriminant $\Delta <0$. For any prime $p$, we have
	\begin{equation}
	\label{win}
	w\sum_{(A,B,C)\in \Psi_{p}(a,b,c)}(A,B,C,q)
	=  \left[p-\left(\tf{\Delta}{p}\right)\right](a,b,c,q^{p^{2}}) +(a,b,c,q)-P_{p,0}(a,b,c,q).
	\end{equation}
\end{theorem}

We now prove Theorem \ref{newthm}. In all cases of the proof we start with the left hand side of \eqref{win2}

\begin{equation}
\label{win2}
 w\sum_{(A,B,C)\in \Psi_{p}(a,b,c)}(A,B,C,q)- \left[p-\left(\tf{\Delta}{p}\right)\right](a,b,c,q^{p^{2}})= (a,b,c,q)-P_{p,0}(a,b,c,q),
\end{equation}
and using the results of the previous sections, we end with the right hand side of \eqref{win2}. The proof is split according to the sign of $\left(\frac{\Delta}{p}\right)$ and if $p\mid a$. Throughout the proof we will always take $0\leq h_i <p$.\\

\textbf{\Large Case} $p \mid \Delta, p\mid a$:\\

By Lemma \ref{nplem} we have $|\Psi_p(a,b,c)|=p$ and $(a,bp,cp^2)$ is the only nonprimitive form listed in \eqref{buellh}. Employing Identity \ref{fident} (with $j=0$), Identity \ref{fident2}, and Lemma \ref{pp0thm} we find that in this case we have
\begin{align*}
w\sum_{(A,B,C)\in \Psi_{p}(a,b,c)}(A,B,C,q)- \left[p-\left(\tf{\Delta}{p}\right)\right](a,b,c,q^{p^{2}})&=\sum_{h=0}^{p-1}\sum_{\substack{x\equiv 0\imod{p},\\y\not\equiv 0 \imod{p}}}q^{ax^2 + (b+2ah)xy + (ah^2 +bh +c)y^2}\\
&=\sum_{\substack{x,\\y\not\equiv 0 \imod{p}}}q^{ax^2 + bxy + cy^2}\\
&=(a,b,c,q)-P_{p,0}(a,b,c,q),
\end{align*}
as desired.\\

\textbf{\Large Case} $p \mid \Delta, p\nmid a$:\\

By Lemma \ref{nplem}, the only nonprimitive form in \eqref{buellh} is $(ap^2, p(b+2ah_1), ah_{1}^2 + bh_1 +c,q)$, where for $p$ odd we have $h_1\equiv \tf{-b}{2a} \imod{p}$ and for $p=2$ we have $h_1 \equiv c \imod{2}$. Employing Identity \ref{fident} (with $j=0$), the left hand side of \eqref{win2} is 
\begin{equation}
\label{cry3}
\sum_{\substack{x\not\equiv 0\imod{p},\\y\equiv 0 \imod{p}}}q^{ax^2 +bxy+cy^2} + \sum_{\substack{h=0,\\h\not\equiv h_1 \imod{p}}}^{p-1}\sum_{\substack{x\equiv 0\imod{p},\\y\not\equiv 0 \imod{p}}}q^{ax^2 + (b+2ah)xy + (ah^2 +bh +c)y^2}.
\end{equation}
Employing Identity \ref{fident2}, we see \eqref{cry3} becomes
\begin{equation}
\label{rr2}
\sum_{\substack{x\not\equiv 0\imod{p},\\y\equiv 0 \imod{p}}}q^{ax^2 +bxy+cy^2}+\sum_{\substack{x,\\y\not\equiv 0 \imod{p}}}q^{ax^2 + bxy + cy^2}-\sum_{\substack{x\equiv 0\imod{p},\\y\not\equiv 0 \imod{p}}}q^{ax^2 + (b+2ah_1)xy + (ah_1^2 +bh_1 +c)y^2}.
\end{equation}
Employing Identity \ref{fident} transforms \eqref{rr2} into
\begin{equation}
\label{rl}
\sum_{\substack{x,\\y\equiv 0 \imod{p}}}q^{ax^2 +bxy+cy^2}+\sum_{\substack{x\equiv 0\imod{p},\\y\equiv 0 \imod{p}}}q^{ax^2 +bxy+cy^2}-(ap^2, p(b+2ah_1), ah_{1}^2 + bh_1 +c,q).
\end{equation}
It is clear that \eqref{rl} is
\begin{equation}
\label{cry2}
(a,b,c,q)-P_{p,0}(a,b,c,q),
\end{equation}
and we have finished this case.\\

\textbf{\Large Case} $\left(\frac{\Delta}{p}\right)=-1$:\\

By Lemma \ref{nplem} all forms of \eqref{buellh} are primitive in this case. Employing Identity \ref{fident} (with $j=0$) the left hand side of \eqref{win2} is  
\begin{equation}
\label{sq}
\sum_{\substack{x\not\equiv 0\imod{p},\\y\equiv 0 \imod{p}}}q^{ax^2 +bxy+cy^2} + \sum_{h=0}^{p-1}\sum_{\substack{x\equiv 0\imod{p},\\y\not\equiv 0 \imod{p}}}q^{ax^2 + (b+2ah)xy + (ah^2 +bh +c)y^2}.
\end{equation}
Employing Identity \ref{fident2}, we see \eqref{sq} becomes
\begin{equation}
\label{sq2}
\sum_{\substack{x\not\equiv 0\imod{p},\\y\equiv 0 \imod{p}}}q^{ax^2 +bxy+cy^2} +\sum_{\substack{x,\\y\equiv 0 \imod{p}}}q^{ax^2 +bxy+cy^2}.
\end{equation}
Adding and subtracting $(a,b,c,q^{p^2})$ and using Lemma \ref{pp0thm} gives that \eqref{sq2} is 
\begin{equation}
\label{sq3}
(a,b,c,q)-P_{p,0}(a,b,c,q),
\end{equation}
and we have finished this case.\\

\textbf{\Large Case} $\left(\frac{\Delta}{p}\right)=1, p\mid a$:\\

By Lemma \ref{nplem} there are two nonprimitive forms in \eqref{buellh}, and they are $(a,pb,cp^2,q),$ and $(ap^2, p(b+2ah_2), ah_{2}^2 + bh_2 +c,q)$, where $h_2\equiv \tf{-c}{b} \imod{p}$ (note this $h_2$ holds for $p=2$ as well). Employing Identity \ref{fident} (with $j=0$) we find the left hand side of \eqref{win2} is  
\begin{equation}
\label{cry322}
\sum_{h=0}^{p-1}\sum_{\substack{x\equiv 0\imod{p},\\y\not\equiv 0 \imod{p}}}q^{ax^2 + (b+2ah)xy + (ah^2 +bh +c)y^2}-\sum_{\substack{x\equiv 0\imod{p},\\y\not\equiv 0 \imod{p}}}q^{ax^2 + (b+2ah_2)xy + (ah_2^2 +bh_2 +c)y^2}.
\end{equation}
Employing Identity \ref{fident2}, we see \eqref{cry322} becomes
\begin{equation}
\label{squ}
\sum_{\substack{x,\\y\not\equiv 0 \imod{p}}}q^{ax^2 + bxy + cy^2}-\sum_{\substack{x\equiv 0\imod{p},\\y\not\equiv 0 \imod{p}}}q^{ax^2 + (b+2ah_2)xy + (ah_2^2 +bh_2 +c)y^2}.
\end{equation}
Adding and subtracting $(a,b,c,q^{p^2})$ and employing Identity \ref{fident} (with $j=0$), we find \eqref{squ} is
\begin{equation}
\label{sque}
\sum_{\substack{x,\\y\not\equiv 0 \imod{p}}}q^{ax^2 + bxy + cy^2}+(a,b,c,q^{p^2})-\sum_{\substack{x\equiv 0\imod{p},\\y}}q^{ax^2 + (b+2ah_2)xy + (ah_2^2 +bh_2 +c)y^2}.
\end{equation}
Lastly we add and subtract $(a,bp,cp^2,q)$ in \eqref{sque} to yield
\begin{equation}
\label{cry222}
(a,b,c,q)-(a,pb,cp^2,q)-(ap^2, p(b+2ah_2), ah_{2}^2 + bh_2 +c,q)+(a,b,c,q^{p^2}),
\end{equation}
and applying Lemma \ref{pp0thm} finishes this case.\\

\textbf{\Large Case} $\left(\frac{\Delta}{p}\right)=1, p\nmid a$:\\

By Lemma \ref{nplem} there are two nonprimitive forms in \eqref{buellh}, and they are $(ap^2, p(b+2ah_3), ah_{3}^2 + bh_3 +c,q),$ and $(ap^2, p(b+2ah_4), ah_{4}^2 + bh_4 +c,q)$, where for $p$ odd we take $h_3\equiv \tf{-b+\sqrt{\Delta}}{2a} \imod{p}$, $h_4\equiv \tf{-b-\sqrt{\Delta}}{2a} \imod{p}$, and for $p=2$ we can simply take $h_3\not\equiv h_4 \imod{2}$. Employing Identity \ref{fident} (with $j=0$) along with Identity \ref{fident2} shows the left hand side of \eqref{win2} to be 
\begin{equation}
\label{cry3222}
\sum_{\substack{x\not\equiv 0\imod{p},\\y\equiv 0 \imod{p}}}q^{ax^2 +bxy+cy^2} + \sum_{\substack{x,\\y\not\equiv 0 \imod{p}}}q^{ax^2 + bxy + cy^2}- \sum_{i=3}^4\sum_{\substack{x\equiv 0\imod{p},\\y\not\equiv 0 \imod{p}}}q^{ax^2 + (b+2ah_i)xy + (ah_{i}^2 +bh_i +c)y^2}.
\end{equation}
Adding and subtracting $2(a,b,c,q^{p^2})$ and employing Identity \ref{fident} (with $j=0$), \eqref{cry3222} becomes
\begin{equation}
\label{cry2222}
(a,b,c,q)-(ap^2, p(b+2ah_3), ah_{3}^2 + bh_3 +c,q)-(ap^2, p(b+2ah_4), ah_{4}^2 + bh_4 +c,q)+(a,b,c,q^{p^2}).
\end{equation}
Applying Lemma \ref{pp0thm} finishes this case, and the theorem is proven.\\

\section{Relating Theorem \ref{newthm} to \cite[Theorem 5.1]{patane} }

\label{conc}

In this section we use Theorem \ref{newthm} to prove Theorem 5.1 of \cite{patane}. First we give an example to illustrate the difference between Theorem \ref{newthm} and \cite[Theorem 5.1]{patane}. In Section \ref{bue} we discuss the map $\Psi_3$ between the class groups CL$(-55)$ and CL$(-55\cdot 3^2)$. We continue this example by examining one of the identities of Theorem \ref{newthm} with $\Delta=-55$ and $p=3$. Theorem \ref{newthm} yields 
\begin{equation}
	\label{thmid}
  (1,1,124,q)+ (9,9,16,q)+ 2(9,3,14,q)=4(1,1,14,q^9)+P_{3,1}(1,1,14,q)+P_{3,2}(1,1,14,q).
	\end{equation}
	
In general, Theorem \ref{newthm} yields identities which are dissections modulo $p$ of the theta series on the left hand side of \eqref{win}. Equation \eqref{thmid} is a dissection modulo 3 of $(1,1,124,q)+ (9,9,16,q)+ 2(9,3,14,q)$. Furthermore we see the forms $(1,1,124), (9,9,16)$ share a genus which is different than the genus containing $(9,3,14)$. See Section \ref{bue} for the genus structure of CL$(-55\cdot 3^2)$. The forms $(1,1,124), (9,9,16)$ are in a different genus than $(9,3,14)$ because they have different assigned character values for the character $\left(\tf{\bullet}{3}\right)$. In other words, if $(1,1,124;n)+ (9,9,16;n) >0$ and $3\nmid n$ then $n\equiv 1 \imod{3}$. Similarly if $(9,3,14;n) >0$ and $3\nmid n$ then $n\equiv 2 \imod{3}$. Employing Lemma \ref{plemma} along with the above discussion allows us to separate \eqref{thmid} into the two identities
\begin{equation}
	\label{thmid2}
	\begin{aligned}
  (1,1,124,q)+ (9,9,16,q)&=2(1,1,14,q^9)+P_{3,1}(1,1,14,q),\\
	2(9,3,14,q)&=2(1,1,14,q^9)+P_{3,2}(1,1,14,q).
	\end{aligned}
	\end{equation}
Theorem 5.1 of \cite{patane} directly claims the identities of \eqref{thmid2}. Theorem 5.1 of \cite{patane} is Theorem \ref{newthm} with the addition that we consider the congruence conditions implied by the assigned characters of the genera. An example is when the left hand side of \eqref{win} contains theta series associated to forms of two genera, we break \eqref{win} into two identities whose sum is \eqref{win}. We now state \cite[Theorem 5.1]{patane}.

\begin{theorem}
	\label{newthm2}
	 Let $(a,b,c)$ be a primitive form of discriminant $\Delta$, and $G$ a genus of discriminant $\Delta p^2$ with $\Psi_{G,p}(a,b,c)$ nonempty. For $p$ an odd prime, we have
\begin{equation}
\label{genp}
 w\sum_{(A,B,C)\in \Psi_{G,p}(a,b,c)}(A,B,C,q)
=   w|\Psi_{G,p}(a,b,c)|(a,b,c,q^{p^{2}}) +\sum_{i=1}^{p-1} \tf{\left(\tf{ri}{p}\right) +1}{2}P_{p,i} (a,b,c,q)
\end{equation}

 and for $p=2$,
\begin{equation}
\label{genp2}
 w\sum_{(A,B,C)\in \Psi_{G,2}(a,b,c)}(A,B,C,q)
=  w|\Psi_{G,2}(a,b,c)|(a,b,c,q^{4}) + P_{2^{t+1},r} (a,b,c,q) 
\end{equation}
\noindent
where\\
\[
	w :=\left\{ \begin{array}{ll}
        3&  \Delta  =-3,\\
				   2&  \Delta =-4,\\
					   1&  \Delta  <-4,\\
     \end{array}
     \right.
	\]
	$r$ is coprime to $\Delta p^2$ and is represented by any form of $\Psi_{G,p}(a,b,c)$. When $\Delta \equiv 0 \imod{16}$ we define $t=2$, and for $\Delta \not\equiv 0 \imod{16}$ we define $t=0,1$ according to whether $\Delta$ is odd or even.
\end{theorem}
Here $\Psi_{G,p}(a,b,c):= \Psi_{p}(a,b,c) \cap G$, and all other notation is consistent with our notation. We note that the coefficient $\tf{\left(\frac{ri}{p}\right) +1}{2}$ of $P_{p,i}(a,b,c,q)$ is simply the integer 0 or 1 depending on the congruence class of $ri$.

\begin{proof}
Our proof naturally splits according to the parity of $p$ and whether $p \mid \Delta$. In general, both \eqref{genp} and \eqref{genp2} are dissections modulo $p$ of the theta series
\[
\sum_{(A,B,C)\in \Psi_{G,p}(a,b,c)}(A,B,C,q).
\]
Lemma \ref{plemma} shows 
\[
P_{p,0}\sum_{(A,B,C)\in \Psi_{G,p}(a,b,c)}(A,B,C,q)= |\Psi_{G,p}(a,b,c)|(a,b,c,q^{p^{2}}),
\]
which is the $0$ modulo $p$ dissection of \eqref{genp} and \eqref{genp2}. Our proof now breaks into cases.\\

\noindent
\large\textbf{Case $p$ odd, $p\nmid \Delta$:}\\

Theorem \ref{newthm} gives the identity
\begin{equation}
\label{er}
w\sum_{i=1}^{p-1}\sum_{(A,B,C)\in \Psi_{p}(a,b,c)}P_{p,i}(A,B,C,q)= \sum_{i=1}^{p-1} P_{p,i} (a,b,c,q).
\end{equation}
In the case when $p$ is odd and $p\nmid \Delta$ we know from Section \ref{bue} that $\Psi_{p}(a,b,c)$ is split equally over two genera which have the same assigned characters except for the character $\left(\tf{\bullet}{p}\right)$. Let $G_1$ be the genus with assigned character $\left(\tf{\bullet}{p}\right)=1$, $G_2$ the genus with assigned character $\left(\tf{\bullet}{p}\right)=-1$, and $\Psi_{p}(a,b,c)$ is contained in $G_1 \cup G_2$. The left hand side of \eqref{er} is
\begin{equation}
\label{er2}
w\sum_{i=1}^{p-1}\sum_{(A,B,C)\in \Psi_{G_1,p}(a,b,c)}P_{p,i}(A,B,C,q)+ w\sum_{i=1}^{p-1}\sum_{(A,B,C)\in \Psi_{G_2,p}(a,b,c)}P_{p,i}(A,B,C,q).
\end{equation}
The right hand side of \eqref{er} is
\begin{equation}
\label{er3}
\sum_{\substack{i=1\\  \left(\tf{i}{p}\right)=1}}^{p-1} P_{p,i} (a,b,c,q) +\sum_{\substack{i=1\\  \left(\tf{i}{p}\right)=-1}}^{p-1} P_{p,i} (a,b,c,q).
\end{equation}
If $(A,B,C)\in G_1$ and $(A,B,C;r)>0$ for some $r$ coprime to $\Delta p^2$, then $(A,B,C;n)=0$ for any $n$ with $\left(\tf{n}{p}\right)=-1$. Similarly if $(A,B,C)\in G_2$ and $(A,B,C;r)>0$ for some $r$ coprime to $\Delta p^2$, then $(A,B,C;n)=0$ for any $n$ with $\left(\tf{n}{p}\right)=1$. We arrive at the identities
\begin{equation}
\label{er40}
\begin{aligned}
w\sum_{i=1}^{p-1}\sum_{(A,B,C)\in \Psi_{G_1,p}(a,b,c)}P_{p,i}(A,B,C,q)&=\sum_{\substack{i=1\\  \left(\tf{i}{p}\right)=1}}^{p-1} P_{p,i} (a,b,c,q),\\
w\sum_{i=1}^{p-1}\sum_{(A,B,C)\in \Psi_{G_2,p}(a,b,c)}P_{p,i}(A,B,C,q)&=\sum_{\substack{i=1\\  \left(\tf{i}{p}\right)=-1}}^{p-1} P_{p,i} (a,b,c,q),\\
\end{aligned}
\end{equation}
which shows Theorem \ref{newthm2} when $p$ is odd and $p\nmid \Delta$.\\

\noindent
\textbf{\large Case $p$ odd, $p\mid \Delta$:}\\

We now show Theorem \ref{newthm2} when $p$ is odd and $p\mid \Delta$. In this case, $\Psi_{p}(a,b,c) \subseteq G$. Since $p$ is odd and $p\mid \Delta$, the character $\left(\tf{\bullet}{p}\right)$ is one of the assigned characters for the discriminant $\Delta$. If $(a,b,c)\in$CL$(\Delta)$ is in a genus $g$ which has $\left(\tf{\bullet}{p}\right)=1$ then $P_{p,r}(a,b,c,q)=0$ for any $r$ with $\left(\tf{r}{p}\right)=-1$. In this case, showing \eqref{genp} is equivalent to showing
\begin{equation}
\label{genpe}
 w\sum_{(A,B,C)\in \Psi_{G,p}(a,b,c)}(A,B,C,q)
=   w|\Psi_{p}(a,b,c)|(a,b,c,q^{p^{2}}) +\sum_{i=1}^{p-1} P_{p,i} (a,b,c,q),
\end{equation}
where we used $\Psi_{G,p}(a,b,c)=\Psi_{p}(a,b,c)$ and $P_{p,r}(a,b,c,q)=0$ for any $r$ with $\left(\tf{r}{p}\right)=-1$. Equation \eqref{genpe} is exactly \eqref{win}. The case when $(a,b,c)\in$CL$(\Delta)$ is in a genus $g$ with assigned character $\left(\tf{\bullet}{p}\right)=-1$ follows similarly.\\

\noindent
\textbf{\large Case $p=2$, $p\nmid \Delta$:}\\

When $p=2$ and $p\nmid \Delta$ then \eqref{genp2} becomes
\begin{equation}
 w\sum_{(A,B,C)\in \Psi_{G,2}(a,b,c)}(A,B,C,q)
=  w|\Psi_{G,2}(a,b,c)|(a,b,c,q^{4}) + P_{2,1} (a,b,c,q),
\end{equation}
which is equivalent to \eqref{win}.\\

\noindent
\large\textbf{Case $p=2$, $p\mid \Delta$:}\\

Lastly we have the case when $p=2$ and $\Delta$ is even. Due to the nature of the assigned characters of a genus, there are several cases to consider when $p=2$ and $\Delta$ is even. This is apparent from \eqref{numgenform} as well as examining whether the characters $\delta:=\left(\tf{-1}{r}\right)$, $\epsilon:=\left(\tf{2}{r}\right)$, and $\delta \epsilon:=\left(\tf{-2}{r}\right)$, are part of the assigned character list for $\Delta$ and $4\Delta$. Details regarding the congruence conditions when $\Delta$ contains the assigned characters $\delta$, $\epsilon$, and $\delta\epsilon$ are given in \cite{buell} and \cite{cox}. The assigned characters for even discriminants are given in Table 2, with $\chi_i:=\left(\tf{\bullet}{p_i}\right)$ and $p_i$ is an odd prime dividing $\Delta$ where $i$ runs up to the number of distinct odd primes dividing $\Delta$.

\begin{table}
\label{rrr}
\caption{}
\begin{center}
\begin{tabular}{ |l |c| }
  \hline
	$\Delta$ & assigned characters  \\ \hline
  $\Delta \equiv 4 \imod{16}$ & $\chi_1,\ldots,\chi_r$ \\ \hline
	$\Delta \equiv 12 \imod{16}$ & $\chi_1,\ldots,\chi_r,\delta$ \\ \hline
	$\Delta \equiv 24 \imod{32}$ & $\chi_1,\ldots,\chi_r,\delta\epsilon$ \\ \hline
	$\Delta \equiv 8 \imod{32}$ & $\chi_1,\ldots,\chi_r,\epsilon$ \\ \hline
	$\Delta \equiv 16 \imod{32}$ & $\chi_1,\ldots,\chi_r,\delta$ \\ \hline
	$\Delta \equiv 0 \imod{32}$ & $\chi_1,\ldots,\chi_r,\delta, \epsilon$ \\ \hline
\end{tabular}
\end{center}
\end{table}

 Our proof now splits according to whether $\tf{v(\Delta p^2)}{v(\Delta)}=1,2$ along with congruence conditions on $\Delta$. In all of these cases we have $|\Psi_2(a,b,c)|=2$ unless $\Delta=-4$. If $\Delta=-4$ then Theorem \ref{newthm2} directly reduces to Theorem \ref{newthm} which reduces to the main theorem of \cite{patane} since both $-4$ and $-16$ are idoneal discriminants. \\

We first consider the case when $\tf{v(\Delta p^2)}{v(\Delta)}=1$ which implies $\Psi_2$ maps into a single genus. Hence Theorem \ref{newthm2} reduces to Theorem \ref{newthm} if can show
\begin{equation}
\label{fbreak}
 P_{2^{t+1},r} (a,b,c,q)
=  P_{2,1} (a,b,c,q),
\end{equation}
where $t$ is given in Theorem \ref{newthm2} and $r$ is coprime to $2\Delta$ and represented by $(a,b,c)$. Equation \eqref{numgenform} implies that we need to consider $\Delta \equiv 0 \imod{32}$, or $\Delta \equiv 12 \imod{16}$. When $\Delta \equiv 0 \imod{32}$, \eqref{fbreak} becomes
\begin{equation}
\label{fbreak2}
 P_{8,r} (a,b,c,q)
=  P_{2,1} (a,b,c,q).
\end{equation}
Equation \eqref{fbreak2} is equivalent to showing $(a,b,c;s)=0$ for all odd $s$ coprime to $\Delta$ and $s\not\equiv r \imod{8}$. This congruence condition follows from the fact that when $\Delta \equiv 0 \imod{32}$, both $\Delta$ and $4\Delta$ have the same assigned characters which are $\chi_p, \delta, \epsilon$ for all odd primes $p\mid \Delta$.

Similarly when $\Delta \equiv 12 \imod{16}$ then \eqref{fbreak} becomes
\begin{equation}
\label{fbreak3}
 P_{4,r} (a,b,c,q)
=  P_{2,1} (a,b,c,q).
\end{equation}
Equation \eqref{fbreak3} is equivalent to showing $(a,b,c;s)=0$ for all odd $s$ coprime to $\Delta$ and $s\not\equiv r \imod{4}$. This congruence condition follows from the fact that when $\Delta \equiv 12 \imod{16}$, both $\Delta$ and $4\Delta$ have the same assigned characters which are $\chi_p, \delta$ for all odd primes $p\mid \Delta$.\\

We are now left with the cases which all have $\tf{v(\Delta p^2)}{v(\Delta)}=2$, and so $\Psi_2$ consists of two forms in different genera. In other words, we are left with the cases in which $\Delta p^2$ has exactly one additional character than $\Delta$. Examining Table 2, we see these are the cases when $\Delta \equiv 4 \imod{16}$ and $\Delta \equiv 8,16,24 \imod{32}$. Let us call this one additional character $\lambda$. For example, when $\Delta \equiv 4 \imod{16}$ the assigned characters of $\Delta$ are $\chi_1,\ldots,\chi_m$ and the assigned characters of $4\Delta$ are $\chi_1,\ldots,\chi_m, \delta$. In this example, $\lambda$ would be the character $\delta$. By taking $\lambda$ to be a general character we can prove the remaining cases together.\\

Fix $(a,b,c)\in$CL$(\Delta)$. Let $G_1$ be the genus of $4\Delta$ with assigned character $\lambda=1$, $G_2$ the genus of $4\Delta$ with assigned character $\lambda=-1$, and $\Psi_{2}(a,b,c)=\{(A,B,C), (D,E,F)\}$ so that $(A,B,C) \in G_1$ and $(D,E,F)\in G_2$. Theorem \ref{newthm} gives 
\begin{equation}
\label{que}
(A,B,C,q)+ (D,E,F,q) = 2(a,b,c,q^4)+P_{2,1}(a,b,c,q).
\end{equation}
We can write $P_{2,1}(a,b,c,q)=P_{2^k,r_1}(a,b,c,q) + P_{2^k,r_2}(a,b,c,q)$ where $\lambda(r_1)=1$, $\lambda(r_2)=-1$, and $k$ is 2 or 3 depending on the character $\lambda$. Employing Lemma \ref{plemma} yields the identities
\begin{equation}
\label{que2}
\begin{aligned}
(A,B,C,q) = (a,b,c,q^4)+P_{2^k,r_1}(a,b,c,q),\\
(D,E,F,q) = (a,b,c,q^4)+P_{2^k,r_2}(a,b,c,q).\\
\end{aligned}
\end{equation}
The identities given in \eqref{que2} are the identities of Theorem \ref{newthm2}, and we have finished the proof of Theorem \ref{newthm2}.

\end{proof}

\section{Lambert Series and Product Representation Formulas}
\label{examples}

One of the main applications of Theorem \ref{newthm2} is that we are often able to deduce a Lambert series decomposition of the left hand side of \eqref{genp} and \eqref{genp2}, and hence a product representation formula for the associated forms. Theorem \ref{newthm2} yields a Lambert series decomposition only when the theta series on the left hand side of \eqref{genp} and \eqref{genp2} are associated to the entire genus. We illustrate this property with the example of $\Delta=-23$ and $p=3$.\\

Let $\Delta=-23$ and $p=3$. The class group and genus structure for the relevant discriminants is given by

\begin{center}
\begin{tabular}{ | c | c | }
  \hline     
  CL$(-23) \cong \mathbb{Z}_3$& $\left(\tf{r}{23}\right)$\\
  \hline                   
   $(1,1,6)$, $(2,1,3)$, $(2,-1,3)$ &$+1$  \\ \hline 
\end{tabular}
\begin{flushright}
			 
			 \end{flushright}
   \end{center}

\begin{center}
\begin{tabular}{ | c | c |c| }
  \hline     
  CL$(-207) \cong \mathbb{Z}_6$& $\left(\tf{r}{23}\right)$& $\left(\tf{r}{3}\right)$ \\
  \hline                   
   $(1,1,52)$, $(4,1,13)$, $(4,-1,13)$ &$+1$ & $+1$ \\ \hline 
	$(8,7,8)$, $(2,1,26)$, $(2,-1,26)$ &$+1$ & $-1$ \\ \hline 
\end{tabular}
\begin{flushright}
			 .
			 \end{flushright}
   \end{center}

We compute
	\begin{equation}
	\label{psihelpp2}
	\begin{aligned}
  \Psi_{3}(1,1,6)&=\{(1,1,52), (8,7,8)\},\\
	\Psi_{3}(2,1,3)&=\{(2,-1,26), (4,1,13)\},\\
	\Psi_{3}(2,-1,3)&=\{(2,1,26), (4,-1,13)\}.\\
	\end{aligned}
	\end{equation}

Employing Theorem \ref{newthm} yields the identities
\begin{equation}
\label{nopenopenope}
\begin{aligned}
(1,1,52,q)+ (8,7,8,q)&=2(1,1,6,q^9)+ (P_{3,1}+P_{3,2})(1,1,6,q),\\
(2,1,26,q)+ (4,1,13,q)&=2(2,1,3,q^9)+ (P_{3,1}+P_{3,2})(2,1,3,q).\\
\end{aligned}
\end{equation}
Either by Theorem \ref{newthm2} or by employing congruences directly to \eqref{nopenopenope}, we find
\begin{equation}
\label{nopers}
\begin{aligned}
(1,1,52,q)&=(1,1,6,q^9)+ P_{3,1}(1,1,6,q),\\
(8,7,8,q)&=(1,1,6,q^9)+ P_{3,2}(1,1,6,q),\\
(4,1,13,q)&=(2,1,3,q^9)+ P_{3,1}(2,1,3,q),\\
(2,1,26,q)&=(2,1,3,q^9)+ P_{3,2}(2,1,3,q).\\
\end{aligned}
\end{equation}
The identities of \eqref{nopenopenope} and \eqref{nopers} do not directly yield Lambert series decompositions since the left hand sides are not associated with the entire genus. However we can combine the identities of \eqref{nopers} to find
\begin{equation}
\label{nopers2}
\begin{aligned}
(1,1,52,q)+2(4,1,13,q)&=f(q^9) +P_{3,1}f(q),\\
(8,7,8,q)+2(2,1,26,q)&=f(q^9) +P_{3,2}f(q),\\
\end{aligned}
\end{equation}
where $f(q)=(1,1,6,q)+2(2,1,3,q)$ is the theta series associated with the principal genus of $\Delta$. The identities of \eqref{nopers2} yield Lambert series decompositions and we demonstrate how to derive these Lambert series and product representation formulas.\\

Dirichlet's formula for quadratic forms gives $f(q)$ as a Lambert series
\begin{equation}
\label{d}
f(q):=(1,1,6,q)+2(2,1,3,q)=3+2\sum_{n=1}^{\infty}\left(\tf{-23}{n}\right)\tf{q^{n}}{1-q^{n}}.
\end{equation}
Using \eqref{d} it is not hard to show
\begin{equation}
\label{d2}
(P_{3,1}-P_{3,2})f(q)=2\sum_{n=1}^{\infty}\left(\tf{69}{n}\right)\tf{q^{n}(1-q^n)}{1-q^{3n}}.
\end{equation}
For convenience we define the Lambert series
\begin{equation}
\label{yes}
L_{1}(q)=\sum_{n=1}^{\infty}\left(\tf{-23}{n}\right)\tf{q^{n}}{1-q^{n}},
\end{equation}
and
\begin{equation}
\label{yeash}
L_2(q)=\sum_{n=1}^{\infty}\left(\tf{69}{n}\right)\tf{q^{n}(1-q^n)}{1-q^{3n}}.
\end{equation}
It is easy to show
\begin{equation}
\label{nopers5}
P_{3,0}L_1(q)=2L_1(q^3)-L_1(q^9).
\end{equation}
Adding and subtracting the identities of \eqref{nopers2} and employing \eqref{d2}, \eqref{nopers5}, gives the Lambert series decompositions for the theta series associated with the genera of discriminant $-207$
\begin{equation}
\label{nopers6}
(1,1,52,q)+2(4,1,13,q)=3+L_1(q)-2L_1(q^3)+3L_1(q^9)+L_2(q),\\
\end{equation}
and
\begin{equation}
\label{nopers7}
(8,7,8,q)+2(2,1,26,q)=3+L_1(q)-2L_1(q^3)+3L_1(q^9)-L_2(q).\\
\end{equation}

Both \eqref{nopers6} and \eqref{nopers7} yield product representation formulas as we now demonstrate. We use the notation
\[
[q^k]\sum_{n\geq 0}a(n)q^n=a(k),
\]
so that $[q^k]f(q)$ is simply the coefficient of $q^k$ in the expansion of the series $f(q)$. The coefficient of $q^n$ in $L_1(q)$ is given by
\[
A(n):=[q^n]\sum_{n=1}^{\infty}\left(\frac{-23}{n}\right) \frac{q^{n}}{1-q^{n}}=\sum_{d|n}\left(\frac{-23}{d}\right).
\]
We see that
\begin{align*}
\sum_{n=1}^{\infty}\left(\frac{69}{n}\right) \frac{q^{n}(1-q^{n})}{1-q^{3n}} &= \sum_{n=1}^{\infty}\sum_{m=0}^{\infty}\left(\frac{69}{n}\right)(q^{n(3m+1)} - q^{n(3m+2)})\\
&=\sum_{n=1}^{\infty}\sum_{m=1}^{\infty}\left(\frac{69}{n}\right)(q^{n(3m-1)} - q^{n(3m-2)})\\
&=\sum_{n=1}^{\infty}\sum_{m=1}^{\infty}\left(\frac{69}{n}\right)\left(\frac{m}{3}\right)q^{nm}\\
&=\sum_{n=1}^{\infty}\left(\sum_{d|n}\left(\frac{69}{d}\right)\left(\frac{n/d}{3}\right) \right)q^{n},
\end{align*}
and so the coefficient of $q^n$ in $L_2(q)$ is given by
\[
B(n):=[q^n]\sum_{n=1}^{\infty}\left(\frac{69}{n}\right) \frac{q^{n}(1-q^{n})}{1-q^{3n}}=\sum_{d|n}\left(\frac{69}{d}\right)\left(\frac{n/d}{3}\right).
\]

It is easy to check that for a prime $p$
\begin{equation}
\label{au}
A(p^{\alpha}) = \left\{ \begin{array}{ll}
       1 &  p=23, \\
       1+\alpha & \left(\frac{-23}{p}\right)=1, \\
       \frac{(-1)^{\alpha}+1}{2} &  \left(\frac{-23}{p}\right)=-1,
     \end{array}
     \right.
\end{equation}
and
\begin{equation}
\label{bdub}
B(p^{\alpha}) = \left\{ \begin{array}{ll}
       0 &  p=3, \alpha \neq 0 \\
       (-1)^{\alpha} &  p=23,\\
       1+\alpha & \left(\tf{-23}{p}\right)=1,\mbox{ and }\left(\tf{p}{3}\right)=1, \\
       (-1)^{\alpha}(1+\alpha) &  \left(\frac{-23}{p}\right)=1,\mbox{ and }\left(\frac{p}{3}\right)=-1, \\
       \frac{(-1)^{\alpha}+1}{2} &  \left(\tf{-23}{p}\right)=-1.
     \end{array}
     \right.
\end{equation}
Since $A(n)$ and $B(n)$ are multiplicative we can use \eqref{au} and \eqref{bdub} along with \eqref{nopers6} and \eqref{nopers7} to give formulas for the number of representations of an integer by a given genus of discriminant $-207$.

\begin{theorem}
\label{jtalk}
Let the prime factorization of $n$ be
\[
n=3^{a}\cdot 23^{b}\prod_{i=1}^{r}p_{i}^{v_{i}}\prod_{j=1}^{s}q_{j}^{w_{j}},
\]
where $p_i\neq3$ and $\left(\tf{-23}{p_i}\right)=1$ and $\left(\tf{-23}{q_j}\right)=-1$. Let
\[
\Lambda(n):=\prod_{i=1}^{r}(1+{v_{i}})\prod_{j=1}^{s}\tf{1+(-1)^{w_{j}}}{2}.
\]
We have
\begin{equation}
\label{h}
(1,1,52;n)+2(4,1,13;n) = \left\{ \begin{array}{ll}
        (1+(-1)^{b+t})\Lambda(n)&  a=0,\\
				0& a=1,\\
				2\Lambda(n)&  a\geq2,\\
     \end{array}
     \right.
\end{equation}
and 
\begin{equation}
\label{h2}
(8,7,8;n)+2(2,1,26;n) = \left\{ \begin{array}{ll}
        (1-(-1)^{b+t})\Lambda(n)&  a=0,\\
				0& a=1,\\
				2\Lambda(n)&  a\geq2,\\
     \end{array}
     \right.
\end{equation}
where $t$ is the number of prime factors $p$ of $n$, counting multiplicity, with $\left(\frac{-23}{p}\right)=1$ and $\left(\frac{p}{3}\right)=-1$.
\end{theorem}

Theorem \ref{jtalk} gives the total number of representations by all forms of a given genus of discriminant $-207$. To find $(a,b,c;n)$ for any particular form of discriminant $-207$, one can employ the techniques of \cite{berpat}. 

Let $A=(2,1,26)$ and $A(q)$ the associated theta series. Recall CL$(-207)\cong \mathbb{Z}_6$, and $A$ is a generator of this group. Theorem \ref{jtalk} gives representation formulas for 
\begin{align}
\label{f1}
I(q)&+2A^2(q),\\
A^3(q)&+2A(q),
\end{align}
where $I$ is the principal form, and $A^k$ corresponds to Gaussian composition $k$ times. The techniques of \cite{berpat} allows us to find representation formulas for
\begin{align}
\label{f2}
I(q)&-A^2(q),\\
A(q)&-A^3(q),
\end{align}
by using the fact that
\begin{align}
   \label{mult1207}   
   M(q):= \df{I(q)-A^2(q)+\left[A(q)-A^3(q)\right]}{2},
   \end{align}
	is an eigenform for all Hecke operators and also employing congruences to separate $I(q)-A^2(q)$ and $A(q)-A^3(q)$. We note that one can use the formulas of Hecke \cite[p.794]{hecke} to show $M(q)$ is an eigenform for all Hecke operators. A concise formula for the action of the Hecke operators on the theta series associated to a binary quadratic form is given by \cite[(1.18)]{berpat}. It is interesting to note that $L_{2}(q)=\tf{I(q)+2A^2(q)-(A^3(q)+2A(q))}{2}$  is an example of a Lambert series which is an eigenform for all Hecke operators.
	
	\cite{berpat} discusses the example CL$(-135)\cong \mathbb{Z}_6 \cong \langle A \rangle$ which is very similar to our example except that for CL$(-135)$ both $\tf{I(q)-A^2(q)\pm \left[A(q)-A^3(q)\right]}{2}$ are eigenforms for all Hecke operators. In our example the combination $\tf{I(q)-A^2(q)-\left[A(q)-A^3(q)\right]}{2}$ is not an eigenform for all Hecke operators and so congruences must be employed to derive \eqref{f2}.
	We do not give explicit representation formulas for \eqref{f2} since the derivation process is similar to the example for $\Delta=-135$ in \cite{berpat}.\\

Another approach to proving Theorem \ref{jtalk} is to employ the general formula \cite[Theorem 8.1, pg. 289]{ref1} proven by Huard, Kaplan, and Williams. This formula gives the total number of representations of an integer $n$ by all the forms in a genus of discriminant $d<0$. Section 8 of \cite{ref2} discusses representations of $n$ by an individual form. We conclude this paper by deriving Theorem \ref{jtalk} from Theorem 8.1 in \cite{ref1}.

In the case that $9\mid n>0$, Theorem 8.1 of \cite{ref1} gives
\begin{align*}
(1,1,52;n)+2(4,1,13;n)&=(8,7,8;n)+2(2,1,26;n)=2\sum_{\mu\mid \tf{n}{9}}\left(\df{-23}{\mu}\right),
\end{align*}
which is consistent with Theorem \ref{jtalk}. When $3 \mid n$ and $9\nmid n$ Theorem 8.1 of \cite{ref1} gives
\begin{align*}
	(1,1,52;n)+2(4,1,13;n)&=(8,7,8;n)+2(2,1,26;n)=0.
\end{align*}
The last case to consider is when $3\nmid n$.  Using the notation of \cite{ref1}, Theorem 8.1 of \cite{ref1} gives the total number of representations by the forms of genus $G$ to be
\[
R_G(n,-207)=\df{1}{2}\sum_{d_1 \in \{1,-3,-23,69\}}\gamma_{d_1}(G)S(n,d_1,\tf{-207}{d_1}),
\]
where
\[
S(n,d_1,\tf{-207}{d_1})=\sum_{\mu \nu=n}\left(\df{d_1}{\mu}\right)\left(\df{-207/d_1}{\nu}\right),
\]
and $\gamma_{d_1}(G)=\left(\tf{d_1}{g}\right)$ with $g$ a positive integer coprime to $d_1$ and represented by the genus $G$. Simplifying yields
\begin{align*}
	(1,1,52;n)+2(4,1,13;n)&=\sum_{\mu\mid n}\left(\df{-23}{\mu}\right)+\sum_{\mu \nu=n}\left(\df{-3}{\mu}\right)\left(\df{69}{\nu}\right),
\end{align*}
and
\begin{align*}
	(8,7,8;n)+2(2,1,26;n)&=\sum_{\mu\mid n}\left(\df{-23}{\mu}\right)-\sum_{\mu \nu=n}\left(\df{-3}{\mu}\right)\left(\df{69}{\nu}\right),
\end{align*}
which is consistent with Theorem \ref{jtalk} in this case.

      \section{Acknowledgments}
   \label{ack}
      
The author would like to thank Alexander Berkovich for valuable discussions. Also the author is grateful to the anonymous referee for helpful comments and for bringing references \cite{ref1} and \cite{ref2} to his attention.

\end{document}